\documentclass[11pt]{amsart}
\usepackage{amssymb,amsthm,amsmath}
\usepackage{extarrows}
\usepackage{marvosym}
\usepackage{empheq}
\usepackage{latexsym}
\usepackage[T1]{fontenc}
\usepackage{upgreek}
\usepackage{hyperref}
\allowdisplaybreaks

\usepackage[margin=1.2in]{geometry}
\title[Graphical partitions and comparability]{Asymptotic bounds on graphical partitions and partition comparability}

\author{Stephen Melczer}
\address{LaCIM, Universit\'e du Qu\'ebec \`a Montr\'eal, CP 8888, Succ. Centre-ville, Montr\'eal QC H3C 3P8, Canada.}
\email{smelczer@uwaterloo.ca}

\author{Marcus Michelen}
\address{University of Illinois at Chicago, Department of Mathematics, Statistics, and Computer Science, 851 S.\ Morgan Street, Chicago, Illinois 60607-7045, USA.}
\email{michelen.math@gmail.com}

\author{Somabha Mukherjee}
\address{Department of Statistics, The Wharton School, University of Pennsylvania, 3730 Walnut Street, Philadelphia, PA 19104-3615, USA.}
\email{somabha@wharton.upenn.edu}

\date{\today}
	
\newtheorem{theorem}{Theorem}
\newtheorem{lemma}[theorem]{Lemma}

\newtheorem*{claim*}{Claim}

\newtheorem{problem}[theorem]{Problem}

\newtheorem{prop}[theorem]{Proposition}

\theoremstyle{definition} 

\newtheorem{question}[theorem]{Question}

\theoremstyle{remark}

\renewcommand{\P}{\mathbf{P}}
\newcommand{\E}{\mathbf{E}}
\newcommand{\eps}{\varepsilon}

\begin{document}
		
\begin{abstract} 
	An integer partition is called graphical if it is the degree sequence of a simple graph. We prove that the probability that a uniformly chosen partition of size $n$ is graphical decreases to zero faster than $n^{-.003}$, answering a question of Pittel.  A lower bound of $n^{-1/2}$ was proven by Erd\H{o}s and Richmond, and so this demonstrates that the probability decreases polynomially.  Key to our argument is an asymptotic result of Pittel characterizing the joint distribution of the first rows and columns of a uniformly random partition, combined with a characterization of graphical partitions due to Erd\H{o}s and Gallai.  Our proof also implies a polynomial upper bound for the probability that two randomly chosen partitions are comparable in the dominance order. 
\end{abstract}

\maketitle 

\section{Introduction}
\label{introduction}

An integer partition is called \emph{graphical} if it can be realized as the size-ordered degree sequence of a simple graph (with no loops or edges).  In his famous 1736 paper on the K\"onigsberg bridge problem, arguably the origin of graph theory, Euler~\cite{Euler1736} gave a necessary condition for a finite sequence of positive integers $(d_1,\dots,d_r)$ to be the degree sequence of a graph: the sum $d_1+\cdots+d_r$ must be even\footnote{In Euler's original words~\cite[\S17]{Euler1736}, `Sequitur ergo ex hac observatione summam omnium pontium, qui in singulas regiones conducunt, esse numerum parem, quia eius dimidium pontium numero aequatur.' [It therefore follows that the total number of bridges leading to the areas must be an even number, because half of it equals the number of bridges.]}. In the nineteenth century, counting the number of graphs with a fixed degree sequence was used to describe the chemical bonds which could be formed between atoms (see, for instance, the 1875 paper of Cayley~\cite{Cayley1875} which concerns the enumeration of certain ``structurally isomeric'' molecules).

Havel~\cite{Havel1955} and Hakami~\cite{Hakimi1962} gave algorithms which take a fixed partition and decide whether or not it is graphical. More recent algorithms of this type were given by Barnes and Savage~\cite{BarnesSavage1995} and Kohnert~\cite{Kohnert2004}.  Seeking to classify graphical partitions, Gale and Ryser~\cite{gale1957theorem,ryser1957combinatorial} showed that given two partitions $\alpha,\beta \vdash n$, there is a \emph{bipartite} graph with one side having degree sequence given by $\alpha$ and the other by $\beta$ if and only if $\alpha \preceq \beta'$.  Here, the conjugate partition $\beta'$ is defined via 
$$\beta_i' = \# \{ j : \beta_j \geq i  \}$$
and the dominance order $\preceq$ is defined by 
\begin{equation}\label{eq:dominance}
\alpha \preceq \beta\ \iff\  \sum_{j = 1}^i \alpha_j \leq \sum_{j = 1}^i \beta_j \text{ for all }i\,.
\end{equation}

In 1960, Erd\H{o}s and Gallai~\cite{ErdosGallai1960} gave a necessary and sufficient condition for a partition to be graphical. Many alternative proofs of this result have been given, including in Harary~\cite{Harary1969}, Berge~\cite{Berge1973}, Choudum~\cite{Choudum1986}, and (having been optimized over the decades to one page of elementary calculations) by Tripathi et al.~\cite{TripathiVenugopalanWest2010}. Sierksma and Hoogeveen~\cite{SierksmaHoogeveen1991} give a proof and list seven equivalent conditions. For our work we use the following statement of the Erd\H{o}s and Gallai characterization.

\begin{lemma}[Erd\H{o}s-Gallai \cite{ErdosGallai1960}]
	\label{lem:graphical}
	A partition $\lambda$ is graphical if and only if 
	\begin{equation}\label{eq:graphical-criterion}
	\sum_{j = 1}^i \lambda_j' \geq \sum_{j = 1}^i \lambda_j + i,\qquad 1 \leq i \leq D(\lambda)
	\end{equation}
	where $D(\lambda) := \max\{k : \lambda_k \geq k \text{ and } \lambda_k' \geq k  \} $ is the size of the Durfee square of $\lambda$ (the largest $k$ such that $\lambda$ contains at least $k$ parts of size $\geq k$).
\end{lemma}

Instead of counting how many partitions are graphical, it is often convenient to study the probability $p(n)$ that a uniformly random partition of size $n$ is graphical\footnote{Recall that Hardy and Ramanujan~\cite{HardyRamanujan1918} determined asymptotics for the total number of partitions of size $n$ as $n\rightarrow\infty$.}. In unpublished remarks from 1982, Wilf conjectured that $p(n)\rightarrow0$ as $n\rightarrow\infty$.   Independently, Macdonald conjectured in 1979 \cite[Ch.\ 1, \S 1, Ex.\ 18]{macdonald1998symmetric}  that if $r(n)$ is the probability that two uniformly random partitions of size $n$ are comparable in the dominance order, then $r(n) \to 0$.  The similarity of the criteria \eqref{eq:dominance} and \eqref{eq:graphical-criterion} suggests that these two seemingly disparate conjectures are in fact quite similar.

Seeking to prove Wilf's conjecture, Erd\H{o}s and Richmond~\cite{ErdosRichmond1993} gave an upper bound of $p(n) \leq .4258$.  Additionally, they used the Erd\H{o}s-Gallai condition \eqref{eq:graphical-criterion} along with the notion of a partition's \emph{rank} to prove that $p(n) \geq 1 - \pi(n-1)/\pi(n)$ for even $n$, where~$\pi(n)$ denotes the number of partitions of $n$. This result, also established in Rousseau and Ali~\cite{RousseauAli1995}, implies that when $n$ is even and sufficiently large then $p(n) \geq \pi/\sqrt{6n}$.  Around the same time, Rousseau and Ali~\cite{RousseauAli1994} also proved a similar upper bound of $p(n) \leq 1/4$ and thus $p(n)$ is bounded.

Both Macdonald's and Wilf's conjectures were finally confirmed in 1999 by Pittel~\cite{Pittel1999}, who analyzed the structure of random partitions and used a Kolmogorov zero-one law to show that $p(n)\to 0$ and $r(n) \to 0$ as $n\rightarrow0$. This work left open the speed with which they approached zero, and whether it was close to the $O(n^{-1/2})$ growth of the known lower bound for $p(n)$. Recently, Pittel~\cite{Pittel2018} returned to this question and showed that for $n$ sufficiently large, $\max\{ r(n),p(n) \} \leq \exp(-c\log(n)/\log\log(n))$ where $c=0.11$. Using numeric values for $p(n)$ computed by Kohnert~\cite{Kohnert2004}, Pittel asked whether this upper bound on $p(n)$ is tight for a value of $c$ near 0.25. Although $p(n)$ decays quite slowly, our primary result shows that this is not the case, and that $p(n)$ decays polynomially:

\begin{theorem}
	\label{thm:main}
	Let $p(n)$ denote the probability that a uniformly randomly chosen partition of $n$ is graphical. Then there exists a constant $C>0$ such that that
	$$p(n) \leq Cn^{-0.003297210314}\,.$$ 
\end{theorem}

A key ingredient is a limit theorem of  Pittel~\cite{Pittel2018} which, for every \mbox{$\varepsilon > 0$}, describes the joint behavior of the first $n^{1/4 - \eps}$ rows \emph{and} columns of a uniformly random permutation.  We will make strong use of this result.

\begin{lemma}[{Pittel~\cite{Pittel2018}}]
	\label{lem:limit-thm}
	Let $\lambda$ be a partition of $n$ chosen uniformly at random.  For each $\gamma < 1/4$, the total variation distance between the joint distribution of the  $\lfloor n^\gamma\rfloor$ tallest columns and the $\lfloor n^\gamma \rfloor$ longest rows of $\gamma$ from the distribution of the random tuple $$\left\{\left\lceil \frac{n^{1/2}}{c} \log \frac{n^{1/2} /c}{ \sum_{j = 1}^i X_j} \right\rceil, \left\lceil\frac{n^{1/2}}{c} \log \frac{n^{1/2}/c }{\sum_{j = 1}^i X_j' } \right\rceil \right\}_{1 \leq i \leq \lfloor n^{\gamma} \rfloor}$$ is $O\left(n^{-1/2 + 2\gamma} \log^3(n)\right)$, where $c := \pi/\sqrt{6}$ and the variables $(X_j)_{j\geq 1}$ and $(X_j')_{j \geq 1}$ are mutually independent exponential random variables with mean~$1$.
\end{lemma}

Using Lemmas \ref{lem:graphical} and \ref{lem:limit-thm} in tandem converts the problem of upper bounding $p(n)$ into a question concerning random walk with exponential increments.  We study the resulting random walk question by comparing to Brownian motion and analyzing the Gaussian process that emerges.  By similar methods, we also demonstrate a polynomial upper bound for the probability that two partitions are comparable in the dominance order:

\begin{theorem}
	\label{thm:comparable}
	Let $r(n)$ denote the probability that for two uniformly randomly chosen partitions $\lambda,\mu$ of $n$ we have $\lambda \preceq \mu$.  Then there is a constant $C >0$ so that $$ r(n) \leq C n^{-0.003297210314}\,.$$
\end{theorem}

Macdonald's 1979 conjecture and Pittel's result that $r(n) \to 0$ sheds light on the typical behavior of Kostka numbers.   Recall that given two partitions $\lambda,\mu \vdash n$, the \emph{Kostka number} $K_{\lambda,\mu}$ is the number of semistandard Young tableaux of shape $\lambda$ and weight $\mu$; for equivalent definitions in terms of Schur polynomials or the representation theory of the symmetric group, see the standard references \cite{macdonald1998symmetric,sagan2013symmetric,stanley}.  Clearly it is the case that $K_{\lambda,\mu} \geq 0$; it is a standard exercise to show that positivity coincides precisely with comparability \cite[Theorem I.6.5]{macdonald1998symmetric}: $$K_{\lambda,\mu} > 0 \ \iff \ \mu \preceq \lambda\,. $$

Theorem \ref{thm:comparable} shows that the $K_{\lambda,\mu}$ is typically $0$, and only positive with probability decaying polynomially quickly. 

\section{Graphical Partitions and Comparability}
In this section we prove Theorems~\ref{thm:main} and~\ref{thm:comparable}. For $j\geq1$ we always write $X_j$ and~$X_j'$ for mutually independent exponential random variables with mean $1$.

\subsection{Reducing to a random walk bound}
\label{sec:randomwalk}	
Both of the quantities $p(n)$ and $r(n)$ of interest will be upper bounded using exponential random walks. Let
\[ \gamma  \in (0,1/4), \qquad S_j := \sum_{i=1}^j X_i, \qquad \text{and} \qquad S_j' := \sum_{i=1}^j X_i', \] 
and define the shifted quantities $R_j := S_j -j$ and $R_j' := S_j' - j$.  The goal of this subsection is to establish the following bounds:
\begin{prop}
\label{prop:uppersum}
For all large $n$ and $\delta > 0$, we have 
\begin{align*}
	\max\{r(n), p(n)\} &\leq O\left(n^{-1/2 + 2 \gamma} \log^3(n)+n^{-\delta/2}\right) \\
	&\quad+ \P\left(\min_{1\leq \ell \leq \lfloor n^{\gamma}\rfloor}\sum_{j = 1}^\ell \frac{R_j' - R_j}{j} \geq -5 n^{\delta/2} \lceil \log^{3/2}(n) \rceil \right)~.
	\end{align*}
\end{prop}

We split the proof of Proposition~\ref{prop:uppersum} into several smaller steps.  The key idea will be to use Lemmas \ref{lem:graphical} and \ref{lem:limit-thm} to obtain upper bounds on $p(n)$ and $r(n)$ that depend only on $S_j$ and $S_j'$.

\begin{lemma}
\label{lem:firststep}
The following asymptotic upper bound holds:
\begin{align*}
\max\{p(n),r(n)\} \leq \P\left( \bigcap_{i = 1}^{\lfloor n^\gamma \rfloor} \left\{\sum_{j = 1}^i \log\left(\frac{S_j'}{S_j}\right) \geq -1 \right\} \right) + O\left(n^{-1/2 + 2 \gamma} \log^3(n)\right)
\end{align*}
\end{lemma} 
\begin{proof}
	Lemmas~\ref{lem:graphical} and~\ref{lem:limit-thm} imply \begin{align*}p(n) &\leq \P\left( \bigcap_{i = 1}^{\lfloor n^\gamma \rfloor}\left\{\sum_{j = 1}^i \left\lceil \frac{n^{1/2}}{c}\log\frac{n^{1/2}/c}{S_j} \right\rceil \geq \sum_{j = 1}^i \left\lceil \frac{n^{1/2}}{c} \log \frac{n^{1/2}/c}{S_j'}\right\rceil + i\right\} \right)  \\
	&\quad+ O\left(n^{-1/2 + 2\gamma}\log^3(n)\right)\,,
	\end{align*}
	so the inequalities $ x \leq \lceil x \rceil \leq x + 1$ imply $$p(n) \leq \P\left(\bigcap_{i = 1}^{\lfloor n^\gamma\rfloor}\left\{\sum_{j = 1}^i \log\left(\frac{S_j'}{S_j}\right) \geq 0\right\} \right) + O\left(n^{-1/2 + 2 \gamma} \log^3(n)\right)\,.$$
	Similarly, Lemma \ref{lem:limit-thm} implies \begin{align*}
	r(n) &\leq \P\left( \bigcap_{i = 1}^{\lfloor n^\gamma\rfloor}\left\{\sum_{j = 1}^i \left\lceil \frac{n^{1/2}}{c}\log\frac{n^{1/2}/c}{S_j} \right\rceil \geq \sum_{j = 1}^i \left\lceil \frac{n^{1/2}}{c} \log \frac{n^{1/2}/c}{S_j'}\right\rceil\right\} \right) + O\left(n^{-1/2 + 2\gamma}\log^3(n)\right) \\ 
	&\leq \P\left( \bigcap_{i = 1}^{\lfloor n^\gamma\rfloor}\left\{ \sum_{j=1}^i \log\left(\frac{S_j'}{S_j} \right) \geq -1 \right\} \right)+ O\left(n^{-1/2 + 2\gamma}\log^3(n)\right)\,.
	\end{align*}
\end{proof}
 
From here, Proposition \ref{prop:uppersum} will be proven by dealing with $\log(S_j'/S_j)$ in two regimes: for small values of $j$, we will use a Chernoff type-bound to show that these terms are only polynomially large; for large values of $j$, the law of large numbers behavior kicks in, and so $S_j'/S_j$ will be close to $1$, implying that we may use the approximation $\log(1 + x) \approx x$.  

We first bound the summands appearing in Lemma~\ref{lem:firststep} with small index.

\begin{lemma}
\label{lem:union1}
For $\delta > 0$ and $n$ sufficiently large we have
$$\P\left( \bigcup_{j = 1}^{\lceil \log^3(n) \rceil} \left\{\frac{S_j'}{S_j} \geq 1 + n^{\delta/2}/\sqrt{j} \right\}\right) \leq 8 n^{-\delta/2}\,. $$
\end{lemma}

\begin{proof}
For any $\beta > 0$ and $z \in (0,1)$ we have
\begin{align*}
\P\left(\frac{S_j'}{S_j} \geq \beta \right) 
&= \P(e^{z(S_j' - \beta S_j)} \geq 1)  \leq \E \left[e^{z(S_j' - \beta S_j)}\right] = (1-z)^{-j}(1+z\beta)^{-j}\,.
\end{align*}
Choosing $z = (\beta - 1)/(2\beta)$ then yields the bound 
$$\P(S_j'/S_j \geq \beta) \leq \left(1 + \frac{(\beta - 1)^2}{4\beta} \right)^{-j}\,,$$
and substituting $\beta = 1 + n^{\delta/2}/\sqrt{j}$ gives
$$\P\left(S_j'/S_j \geq 1 + n^{\delta/2}/\sqrt{j}\right) \leq \left(1 + \frac{n^{\delta}}{4 j  (1 + n^{\delta/2} /\sqrt{j})} \right)^{-j} \leq \left(\frac{4\sqrt{j}}{n^{\delta/2}} \right)^j\,.$$
If $a_j := (4\sqrt{j}/n^{\delta/2})^j$ then 
$$\sup_{1\leq j \leq \lceil \log^3(n) \rceil} \frac{a_{j+1}}{a_j} = \sup_{1\leq j \leq \lceil \log^3(n) \rceil} 4 \left(\frac{j+1}{j} \right)^{j/2} \frac{\sqrt{j+1}}{n^{\delta/2}} \leq \frac{1}{2}$$ 
for $n$ sufficiently large.  This implies 
\begin{align*}
\sum_{j = 1}^{\lceil \log^3(n) \rceil}\P\left(\frac{S_j'}{S_j} \geq 1 + n^{\delta/2}/\sqrt{j} \right) 
\leq a_1\left(1+1/2+1/4+\cdots\right) = 8 n^{-\delta/2}~,
\end{align*}
and Lemma \ref{lem:union1} follows by the union bound.
\end{proof}

Now, we bound the shift quantities $R_j$ when $j$ is bounded away from the origin.

\begin{lemma}
\label{lem:boundonRj}
For all $j > \log^2(n)$,
$$\P\left(\left|\frac{R_j}{j}\right| \geq \frac{\log(n)}{\sqrt{j}}\right) < 2\exp\left(-\log^2(n)/6\right)\,.$$
\end{lemma}
\begin{proof}
An application of Chernoff bounds yields that, for every~$\delta \in (0,1)$,
\begin{equation}\label{chernoff}
\P\left(\left|\frac{R_j}{j}\right| \geq \delta\right) \leq e^{-j(\delta - \log(1+\delta))} + e^{j(\delta + \log(1-\delta))},
\end{equation}
and the inequality $\log(1+x) < x - x^2/2 + x^3/3$ for all $x > -1$ implies
$$\delta - \log(1+\delta) > \frac{\delta^2}{2} - \frac{\delta^3}{3} > \frac{\delta^2}{6}\quad\text{and}\quad \delta + \log(1-\delta) < -\frac{\delta^2}{2} - \frac{\delta^3}{3} < -\frac{\delta^2}{2}~.$$ 
Hence, \eqref{chernoff} gives
\begin{equation}\label{simp1}
\P\left(\left|\frac{R_j}{j}\right| \geq \delta\right) < 2 e^{-j\delta^2/6}~.
\end{equation}
Taking $\delta = \log(n)/\sqrt{j}$ completes the Lemma.
\end{proof}

Lemma~\ref{lem:boundonRj} allows us to bound, with high probability, summands in Lemma~\ref{lem:union1} by a scaled difference of $R_j'$ and $R_j$.

\begin{lemma}
\label{lem:residual}
With probability at least $1 - 4 n^{1/4} \exp\left(-\log(n)^2/6 \right)$ we have 
$$\sum_{j = \lceil\log^3(n)\rceil+1}^{\ell} \log\left(\frac{S_j'}{S_j}\right) \leq \sum_{j = \lceil\log^3(n)\rceil+1}^{\ell}\frac{R_j' - R_j}{j} + \log^3(n)$$ 
for all $\lceil \log^3(n)\rceil <  \ell \leq \lfloor n^\gamma\rfloor$.
\end{lemma}
\begin{proof}
	
Using the inequality $\log(1 + x) \leq x$ for all $x > -1$, we have
{\small
\begin{equation}\label{singlebound}
\log\left(\frac{S_j'}{S_j}\right) = \log\left(1+\frac{R_j'-R_j}{j+R_j}\right) 
	\leq \frac{R_j'-R_j}{j+R_j}
	= \frac{R_j' - R_j}{j} - \frac{\left(\frac{R_j'}{j} - \frac{R_j}{j}\right)\frac{R_j}{j}}{\left(1 + \frac{R_j}{j}\right)}\,.
\end{equation}
}
Now, for every $n \geq e^4$ and every fixed $j \geq \log^3(n)$ we have, by Lemma \ref{lem:boundonRj},
\begin{equation}\label{lastbound}
\left|\frac{\left(\frac{R_j'}{j} - \frac{R_j}{j}\right)\frac{R_j}{j}}{\left(1 + \frac{R_j}{j}\right)}\right| \leq \frac{2\log^2 (n)/j}{\left|1+R_j/j\right|} \leq \frac{4\log^2(n)}{j}~
\end{equation} 
with probability at least $1-4\exp\left(-\log^2(n)/6\right)$. It follows from equations~\eqref{singlebound} and~\eqref{lastbound} and a union bound that, for all $n\geq e^4$ and for all $\ell$ in the interval $\left(\lceil\log^3(n)\rceil, \lfloor n^{\gamma}\rfloor\right]$ that
\begin{align*}
\sum_{j = \lceil\log^3(n)\rceil+1}^{\ell} \log\left(\frac{S_j'}{S_j}\right) &\leq \sum_{j = \lceil\log^3(n)\rceil+1}^{\ell}\frac{R_j' - R_j}{j} + 4\log^2(n) \sum_{j=2}^{\lfloor n^\gamma \rfloor} \frac{1}{j}\\
&\leq \sum_{j = \lceil\log^3(n)\rceil+1}^{\ell}\frac{R_j' - R_j}{j} + \log^3(n)\,,
\end{align*}
with probability at least  $1 - 4 n^{1/4} \exp\left(-\log(n)^2/3 \right)$. 
\end{proof}

We are now ready to complete the proof of Proposition~\ref{prop:uppersum}.

\begin{proof}[Proof of Proposition~\ref{prop:uppersum}]
Define the three events
\begin{align*}
E_1 &:= \bigcup_{j = 1}^{\lceil \log^3(n) \rceil} \left\{\frac{S_j'}{S_j} \geq 1+ \frac{n^{\delta/2}}{\sqrt{j}} \right\}~, \\[+2mm]
E_2 &:= \bigcup_{\ell=\lceil\log^3(n)\rceil+1}^{\lfloor n^{\gamma}\rfloor} \left\{\sum_{j = \lceil\log^3(n)\rceil+1}^{\ell} \log\left(\frac{S_j'}{S_j}\right)> \sum_{j = \lceil\log^3(n)\rceil+1}^{\ell}\frac{R_j' - R_j}{j} + \log^3(n)\right\}~, \\[+2mm]
E_3 &:= \bigcap_{i=1}^{ \lfloor n^\gamma\rfloor}\left\{\sum_{j = 1}^i \log\left(\frac{S_j'}{S_j}\right) \geq -1\right\}~.
\end{align*} 
By Lemmas \ref{lem:union1} and \ref{lem:residual}, we have $\P(E_1) \leq 8n^{-\delta/2}$ and $\P(E_2) \leq 4 n^{1/4} \exp\left(-\log(n)^2/6 \right)$. Furthermore, on the complement $E_1^c$ we have
\begin{align}\label{onE1c}
\sum_{j = 1}^{\lceil \log^3(n) \rceil} \log\left(\frac{S_j'}{S_j}\right) 
&\leq \sum_{j = 1}^{\lceil \log^3(n) \rceil} \log\left(1+ \frac{n^{\delta/2}}{\sqrt{j}}\right)\nonumber\\ 
&\leq n^{\delta/2}\sum_{j = 1}^{\lceil \log^3(n) \rceil}\frac{1}{\sqrt{j}}\nonumber\\
&\leq n^{\delta/2}\int_0^{\lceil \log^3(n) \rceil} \frac{1}{\sqrt{x}}~dx \nonumber\\
&\leq 2n^{\delta/2}\lceil\log^{3/2}(n)\rceil\,.  
\end{align}
Looking to bound $\P(E_3)$, first note that
\begin{equation}\label{pe3}
\P(E_3) \leq 8n^{-\delta/2} + 4 n^{1/4} \exp\left(-\log(n)^2/6 \right) + \P\left(E_3 \cap E_2^c \cap E_1^c\right)~.
\end{equation}
It follows from \eqref{onE1c}, that
$$E_3 \cap E_1^c \subseteq \bigcap_{i = \lceil \log^3(n) \rceil+1}^{\lfloor n^\gamma \rfloor}\left\{ \sum_{j = \lceil\log^3(n)\rceil+1}^{i} \log\left(\frac{S_j'}{S_j}\right) \geq -3n^{\delta/2}\lceil\log^{3/2}(n)\rceil\right\}~,$$
hence, for all large $n$, we have the containments
\begin{align}
E_3 \cap E_2^c \cap E_1^c &\subseteq \bigcap_{i = \lceil \log^3(n) \rceil+1}^{\lfloor n^\gamma \rfloor} \left\{\sum_{j = \lceil\log^3(n)\rceil+1}^{i}\frac{R_j' - R_j}{j} \geq -3n^{\delta/2}\lceil\log^{3/2}(n)\rceil - \log^3(n)\right\} \nonumber \\
&\subseteq \bigcap_{i = \lceil \log^3(n) \rceil+1}^{\lfloor n^\gamma \rfloor}\left\{\sum_{j = \lceil\log^3(n)\rceil+1}^{i}\frac{R_j' - R_j}{j} \geq -4n^{\delta/2}\lceil\log^{3/2}(n)\rceil \right\}\,. \label{threeint}
\end{align}

From \eqref{pe3} and~\eqref{threeint} and Lemma \ref{lem:firststep} we thus have
\begin{equation}\label{int}
	\max\{p(n),r(n)\} \leq O\left(n^{-1/2 + 2 \gamma} \log^3(n)+n^{-\delta/2}\right) + \P(E)~,
\end{equation}
where $$E := \left\{\min_{\ell \in \left(\lceil\log^3(n)\rceil, \lfloor n^{\gamma}\rfloor\right]}\sum_{j = \lceil \log^3(n)\rceil+1}^\ell \frac{R_j' - R_j}{j} \geq - 4 n^{\delta/2} \lceil \log^{3/2}(n) \rceil\right\}\,.$$
Define another event
$$E_4 := \left\{\min_{1\leq i \leq \lceil \log^3(n)\rceil} \sum_{j=1}^i \frac{R_j'-R_j}{j} \leq -n^{\delta/2}\lceil \log^{3/2}(n) \rceil \right\}, $$
and note that
\begin{equation}\label{sufficient}
E\cap E_4^c \subseteq \left\{\min_{1\leq i \leq \lfloor n^{\gamma}\rfloor}\sum_{j = 1}^i \frac{R_j' - R_j}{j} \geq - 5 n^{\delta/2} \lceil \log^{3/2}(n) \rceil\right\}~.
\end{equation}
In view of \eqref{int} and \eqref{sufficient}, to complete the proof of Proposition~\ref{prop:uppersum} it thus suffices to show that $P(E_4)  = O(n^{-\delta/2})$.  With this in mind, we see
\begin{align*}
	\P(E_4) &\leq \sum_{j=1}^{\lceil \log^3(n)\rceil} \P\left(\left|\frac{R_j'-R_j}{j}\right|\geq \frac{n^{\delta/2}\lceil \log^{3/2}(n) \rceil}{\lceil \log^3(n)\rceil}\right)\\
	&\leq \sum_{j=1}^{\lceil \log^3(n)\rceil} \P\left(\left|\frac{R_j'-R_j}{j}\right|\geq \frac{1}{2}n^{\delta/2}\log^{-3/2}(n)\right)\\
	&\leq 4n^{-\delta}\log^3(n) \sum_{j=1}^{\lceil \log^3(n)\rceil} \frac{\E(R_j'-R_j)^2}{j^2}\\
	&= 8n^{-\delta}\log^3(n) \sum_{j=1}^{\lceil \log^3(n)\rceil} \frac{1}{j}\\ 
	&\leq 8n^{-\delta}\log^3(n)\left(1+\log \lceil \log^3(n)\rceil\right) \\
	&= O(n^{-\delta/2})\,.
	\end{align*}
The proof of Proposition~\ref{prop:uppersum} is now complete.
\end{proof}

\subsection{Comparison to a Brownian Functional}
\label{brownian}
Having established Proposition~\ref{prop:uppersum}, our next step is comparing the random walk $\left\{\sum_{j=1}^\ell \frac{R_j'-R_j}{j}\right\}_{\ell \geq 1}$ to a Gaussian process constructed from a Brownian motion. Towards this, let~$B_t$ be a standard Brownian motion, and define $Z_n = \sum_{k = 1}^{n} B_k/k$ for each non-negative integer $n$.
\begin{prop}\label{prop:exit-time}
Let $\alpha \in [0,\frac{1}{2})$ and let $\beta = 0.01363853235$. Then, as~$N \rightarrow \infty$, we have $$\P\left( \max_{k \leq N}  Z_k \leq N^{\alpha}\right) = O\left(N^{-\beta(1-2\alpha)}\right)\,.$$
\end{prop}

Understanding the covariance structure of $X_t$ is necessary for proving Proposition \ref{prop:exit-time}.

\begin{lemma}\label{covariance}
The process $(Z_n)_{n \geq 0}$ is a mean-zero Gaussian process, and for~$0\leq m \leq n$ we have
$$\mathrm{Cov}(Z_m,Z_n) = 2m -(m+1) H_m + m H_n~,$$
where $H_k = \sum_{j =1}^k 1/j$ is the $k$th harmonic number.
\end{lemma}
\begin{proof}
Fixing $0\leq m \leq n$, we have
\begin{align*}
\mathrm{Cov}(Z_m,Z_n) &= \sum_{i=1}^m \sum_{j=1}^n \frac{\min\{i,j\}}{ij}\\
&= \sum_{i=1}^m \sum_{j=1}^m \frac{\min\{i,j\}}{ij} + \sum_{i=1}^m \sum_{j=m+1}^n \frac{1}{j}\\
&= 2 \sum_{1\leq i<j\leq m} \frac{1}{j} + \sum_{i=1}^m \frac{1}{i} + \sum_{i=1}^m (H_n - H_m)\\
&= 2\sum_{j=2}^m \sum_{i=1}^{j-1} \frac{1}{j} - (m-1)H_m + m H_n\\
&= 2 \sum_{j=2}^m \left(1 - \frac{1}{j}\right) - (m-1)H_m + m H_n\\
&= 2(m-1) -2(H_m-1) -(m-1)H_m + mH_n\\
&= 2m - (m+1)H_m + mH_n ~,
\end{align*}
as desired.
\end{proof}

To establish Proposition \ref{prop:exit-time} we also use the following anti-concentration result of Li and Shao for the supremum of a centered Gaussian process.

\begin{lemma}[{Li and Shao~\cite[Theorem~2.2]{LiShao2004}}]\label{lem:Li-Shao} 
Let $Y_t$ be a mean-zero Gaussian process on $[0,T]$ with $Y_0 = 0$.  For $x > 0$ and $j = 1,\ldots,M$ let $s_j \in [0,T]$ be a sequence such that, for each $i$,
\begin{equation} 
\label{eq:Corr}
\sum_{j = 1}^M \left|\mathrm{Corr}(Y_{s_i},Y_{s_j}) \right| \leq 5/4
\end{equation}
and \begin{equation}
\left(\E Y_{s_i}^2 \right)^{1/2} \geq x/2.
\end{equation}
Then $$\P\left( \max_{t \in [0,T]} Y_t \leq x\right) \leq e^{-M/10}\,.$$
\end{lemma}

We are now ready to prove Proposition \ref{prop:exit-time}.
\begin{proof}[Proof of Proposition \ref{prop:exit-time}] 
Aiming to use Lemma \ref{lem:Li-Shao}, let $s_i := \lceil N^{2\alpha} \rho^i\rceil$ where $\rho > 1$ is a constant to be determined later.  Note that for $N$ large we have $$\left(\E Z_{s_i}^2\right)^{1/2} = (2s_i - H_{s_i})^{1/2} \geq s_i^{1/2} \geq N^{\alpha}\,.$$
For each $\eps > 0$ and $i < j$ we have for sufficiently large $N$ (depending on $\varepsilon$)
\begin{align*}
|\mathrm{Corr}(Z_{s_i},Z_{s_j})| &\leq \frac{\left(2 + H_{s_j} - H_{s_i} \right)s_i}{(2-\eps)\sqrt{s_is_j}} \\
&\leq \frac{\left(2+\varepsilon + (j-i)\log(\rho) \right)}{(2 - \eps)\rho^{(j-i)/2}}\,.
\end{align*}
 Hence, for every $M \geq 1$, we have for sufficiently large $N$ that
\begin{align*}
\sum_{j = 1}^M |\mathrm{Corr}(Z_{s_i},Z_{s_j})| &\leq 1 + 2\sum_{k \geq 1} \frac{\left(2 + \eps + k\log(\rho) \right)}{(2 - \eps)\rho^{k/2}} \,.
\end{align*}
Now, note that 
$$
\frac{2 + \eps}{2 - \eps}\sum_{k \geq 1} \rho^{-k/2} + \frac{\log(\rho)}{2 - \eps}\sum_{k \geq 1}k \rho^{-k/2} 
\leq (1 + O(\eps))\left(\frac{\rho^{-1/2}}{1 - \rho^{-1/2}} + \frac{\log(\rho)}{2} \frac{\rho^{-1/2}}{(1 - \rho^{-1/2})^2} \right)\,.
$$
Choose $\varepsilon > 0$ small enough so that
\begin{equation}\label{sumcorrbound}
\sum_{j = 1}^M |\mathrm{Corr}(Z_{s_i},Z_{s_j})| \leq g(\rho)~,
\end{equation}
where $$g(\rho) :=  1+2 \left(\frac{\rho^{-1/2}}{1 - \rho^{-1/2}} + \frac{\log(\rho)}{2} \frac{\rho^{-1/2}}{(1 - \rho^{-1/2})^2} \right)~.$$
Now, $g(\rho_\ast) = 5/4$ for some $\rho_\ast \approx 1528.691213$, and setting $\rho = \rho_\ast$ yields 
$$\sum_{j = 1}^M |\mathrm{Corr}(Z_{s_i},Z_{s_j})| \leq (1 + O(\eps))\frac{5}{4}\,.$$  
This implies that for $\rho > \rho_\ast$, equation~$\eqref{eq:Corr}$ is satisfied for some $\eps > 0$ for $N$ sufficiently large.  In particular, if we set $\rho = \rho_\ast^{1 + \eps}$ for $\eps > 0$, then equation~\eqref{eq:Corr} is satisfied for $N$ sufficiently large.  Set $M := \lfloor(1 - 2\alpha) \log(N)/\log(\rho) \rfloor$, so that $s_M \leq N$, and hence, $s_1,\ldots,s_M \in [0,N]$. By Lemma \ref{lem:Li-Shao}, we have
\begin{align*}
	\P\left(\max_{t\in[0,N]} Z_t \leq N^{\alpha}\right) &\leq \exp \left(-\frac{1}{10}\left\lfloor\frac{(1 - 2\alpha) \log(N)}{\log(\rho)} \right\rfloor\right)\\ &\leq e^{0.1} N^{-\beta
		(1-2\alpha)}~,
	\end{align*}
where $\beta = 0.01363853235$ for some choice of $\eps > 0$ sufficiently small.
\end{proof}

With Proposition \ref{prop:exit-time} established, all that remains is a comparison between the random walk $\{T_j/\sqrt{2} \}_{j \geq 1}$ where $T_j := R_j'-R_j = \sum_{i = 1}^j (X_i' - X_i)$ and a standard Brownian motion~$B$.  The Kolm\'os-Major-Tusn\'ady (KMT) coupling provides a way to compare simple random walk to Brownian motion \cite{KMT}.  We use a slight modification that applies directly to the case at hand:

\begin{lemma}[{Aurzada and Dereich~\cite{aurzada2013universality}}]
\label{lem:KMT}
There exist constants $\beta_1,\beta_2>0$ such that for each $T \geq e$ there is a coupling of $T_t/\sqrt{2}$ and $B_t$, where
$$\P\left(\sup_{t \in [0,T]} \left|\frac{T_t}{\sqrt{2}} - B_t\right| \geq a \right) \leq e^{-\beta_1 a} T^{\beta_2}$$
for each $a >0$.
\end{lemma}

This allows for a study of $T_j/j$.

\begin{lemma}\label{residualbound}
Suppose that~$\varepsilon, \delta$ and $\gamma$ are positive quantities satisfying \mbox{$\delta + 2\varepsilon<\gamma$}, and let $\beta = 0.01363853235$. Then
$$\P\left(\min_{1 \leq i \leq \lfloor n^\gamma \rfloor} \sum_{j = 1}^i \frac{T_j}{j} \geq - 5  n^{\delta/2} \lceil\log^{3/2}(n)\rceil \right) = O(n^{-\beta(\gamma - \delta-2\varepsilon)})\,.$$
\end{lemma}
\begin{proof}
Fix $\eps > 0$, which will be chosen suitably later. Applying Lemma \ref{lem:KMT} with $a = \log^2(n)$ and $T = \lfloor n^\gamma \rfloor$, we have 
$$\P(F) \leq e^{-\beta_1 \log^2(n)} n^{\gamma \beta_2}\,,$$
where
$$F := \left\{\sup_{t \in [0,\lfloor n^\gamma \rfloor]} \left| \frac{T_t}{\sqrt{2}} - B_t\right| \geq \log^2(n)\right\}~.$$
Now, on $F^c$ we have for sufficiently large $n$ that
\begin{align*}
\sup_{1\leq i \leq \lfloor n^\gamma\rfloor} \left|\sum_{j=1}^i \frac{T_j}{j} - \sqrt{2}\sum_{j=1}^i \frac{B_j}{j}\right| &\leq \sup_{1\leq i \leq \lfloor n^\gamma\rfloor} \sum_{j=1}^i \frac{|T_j - \sqrt{2} B_j|}{j}\\
&\leq\sqrt{2} \log^2(n) \sum_{j=1}^{\lfloor n^\gamma \rfloor} \frac{1}{j}\\
&\leq 2\gamma \log^3(n).
\end{align*}	
Hence, on $F^c$, we have for sufficiently large $n$ that
$$\min_{1 \leq i \leq \lfloor n^\gamma \rfloor} Z_i \geq \frac{1}{\sqrt{2}}\min_{1 \leq i \leq \lfloor n^\gamma \rfloor} \sum_{j = 1}^i \frac{T_j}{j} - \sqrt{2}\gamma \log^3(n)~.$$
Thus, for sufficiently large $n$,
\begin{align*}&\P\left(\min_{1 \leq i \leq \lfloor n^\gamma \rfloor} \sum_{j = 1}^i \frac{T_j}{j} \geq - 5  n^{\delta/2} \lceil\log^{3/2}(n)\rceil \right)\\ &\leq  e^{-\beta_1 \log^2(n)} n^{\gamma \beta_2} + \P\left(\min_{1 \leq i \leq \lfloor n^\gamma \rfloor} Z_i \geq - 3\sqrt{2}  n^{\delta/2} \lceil\log^{3/2}(n)\rceil - \sqrt{2}\gamma \log^3(n)\right)\\
&\leq e^{-\beta_1 \log^2(n)} n^{\gamma \beta_2} + \P\left(\max_{1 \leq i \leq \lfloor n^\gamma \rfloor} Z_i \leq  2 \lfloor n^\gamma \rfloor^{(\delta+2\varepsilon)/(2\gamma)} \right)\\
&\leq e^{-\beta_1 \log^2(n)} n^{\gamma \beta_2} + \lfloor n^\gamma \rfloor ^{-\frac{\beta}{\gamma}(\gamma-\delta-2\varepsilon)}\quad\quad\quad\quad\quad\quad\quad\quad\quad(\textrm{by Proposition \ref{prop:exit-time}})\\&= O\left(n^{-\beta(\gamma-\delta-2\varepsilon)}\right)~,
\end{align*}
as desired.
\end{proof}

We can now complete our proof of Theorems \ref{thm:main} and \ref{thm:comparable}.

\begin{proof}[Proof of Theorems \ref{thm:main} and \ref{thm:comparable}]
It follows from Proposition \ref{prop:uppersum} and Lemma \ref{residualbound} that
\begin{equation}\label{finalbound}
\max\{r(n), p(n)\} \leq O\left(n^{-1/2 + 2 \gamma} \log^3(n)+n^{-\delta/2} + n^{-\beta(\gamma-\delta-2\varepsilon)}\right)~,
\end{equation}
where $\beta = 0.01363853235$ and $\delta, \gamma, \varepsilon$ are any three positive constants, subject to the constraint $\delta+2\varepsilon <\gamma < 1/4$. In order to obtain the most efficient bound for $p(n)$ from \eqref{finalbound}, we thus need to evaluate 
$$\max_{\delta,\gamma,\varepsilon} \min \left\{\frac{1}{2}-2\gamma, \frac{\delta}{2}, \beta(\gamma-\delta-2\varepsilon)\right\}$$ 
subject to the constraints $\delta,\gamma,\varepsilon > 0$ and $\delta+2\varepsilon <\gamma < 1/4$.  We may take $\eps$ arbitrarily close to $0$, thus requiring we maximize $$\max_{\delta,\gamma} \min \left\{\frac{1}{2} - 2\gamma,\frac{\delta}{2},\beta(\gamma - \delta) \right\}$$
subject to $\gamma, \delta > 0, \delta < \gamma < \frac{1}{4}$.  Since setting any of $\gamma = \delta, \delta = 0$ and $\gamma = \frac{1}{4}$ yields a value of zero, the maximum is achieved in the interior of the region that $(\delta,\gamma)$ varies over.  Since an increase in either $\gamma$ or $\delta$ increases one term and decreases another, the maximum must be achieved when all three terms are equal, which implies $\delta = 0.006594420627$ and $\gamma = 0.2483513948$.  Theorems \ref{thm:main} and \ref{thm:comparable} follow.
\end{proof}

\section{Open Problems}

In light of Theorem \ref{thm:main} along with the lower bound of $p(2n) = \Omega(n^{-1/2})$ provided by Erd\H{o}s and Richmond \cite{ErdosRichmond1993}, it remains to closing the gap between the two bounds:

\begin{problem}
	Let $p(n)$ denote the probability that a uniformly random partition of $n$ is graphical.  Does there exist an $\alpha$ so that 
	$$p(2n) = n^{-\alpha + o(1)}$$
	as $n\to\infty$?  If so, what is $\alpha$ equal to?
\end{problem} 

Similarly, the analogous question remains for comparability: 

\begin{problem}
	Let $r(n)$ denote the probability that two uniformly random partitions of $n$ are comparable in the dominance order.  Does there exist a $\beta$ so that 
	$$p(n) = n^{-\beta + o(1)}$$
	as $n\to\infty$?  If so, what is $\beta$ equal to?
\end{problem}

Recalling that $\mu \preceq \lambda \iff K_{\lambda,\mu} > 0$ where $K_{\lambda,\mu}$ are the Kostka numbers, Theorem \ref{thm:comparable}---and Pittel \cite{Pittel1999,Pittel2018} before---shows that Kostka numbers are typically $0$.  As a more open-ended question, we ask for the typical behavior of Kostka numbers conditioned on being positive:

\begin{question}
	Let $\lambda$ and $\mu$ be partitions of $n$ independently chosen uniformly at random.  What is the typical behavior of the conditioned random variable $(K_{\lambda,\mu} \,|\, K_{\lambda,\mu} > 0 )$?
\end{question} 

\section*{Acknowledgments}
We would like to thank Bruce Richmond for bringing our attention to this problem; the authors would also like to thank Robin Pemantle and Cheng Ouyang for useful discussions.

\bibliographystyle{alpha}
\bibliography{bibl}

\end{document}